\documentclass[12pt]{amsart}
\usepackage{amssymb,color}
\numberwithin{equation}{section} 

\usepackage{graphicx}

\newcommand{\bea}{\begin{eqnarray}}
\newcommand{\eea}{\end{eqnarray}}
\newcommand{\ba}{\begin{array}}
\newcommand{\ea}{\end{array}}
\newcommand{\edc}{\end{document}}
\newcommand{\bc}{\begin{center}}
\newcommand{\ec}{\end{center}}
\newcommand{\be}{\begin{equation}}
\newcommand{\ee}{\end{equation}}

\def\bc{{\mathbb C}}

\def\br{{\mathbb R}}

\newtheorem{thm}{Theorem}[section]

\newtheorem{cor}[thm]{Corollary}

\newtheorem{prop}[thm]{Proposition}
\newtheorem{defin}[thm]{Definition}
\theoremstyle{remark}

\begin{document}

\title[The strong "zero-two" law ]{The strong "zero-two" law  for positive contractions of  Banach-Kantorovich $L_p$-lattices}
\author{Inomjon Ganiev}
\address{Inomjon Ganiev\\
 Department of Science in Engineering\\
Faculty of Engineering, International Islamic University Malaysia\\
P.O. Box 10, 50728\\
Kuala-Lumpur, Malaysia}  \email{{\tt inam@iium.edu.my,
ganiev1@rambler.ru}}

\author{Farrukh Mukhamedov}
\address{Farrukh Mukhamedov\\
 Department of Computational \& Theoretical Sciences\\
Faculty of Science, International Islamic University Malaysia\\
P.O. Box, 141, 25710, Kuantan\\
Pahang, Malaysia} \email{{\tt farrukh\_m@iium.edu.my}, {\tt
far75m@yandex.ru}}

\author{Dilmurad Bekbaev}
\address{Dilmurad Bekbaev\\
 Department of Computational \& Theoretical Sciences\\
Faculty of Science, International Islamic University Malaysia\\
P.O. Box, 141, 25710, Kuantan\\
Pahang, Malaysia}

\begin{abstract}
In the present paper we study majorizable operators acting on
Banach-Kantorovich $L_p$-lattices, constructed by a measure $m$
with values in the ring of all measurable functions. Then using
methods of measurable bundles of Banach-Kantorovich lattices, we
prove the strong "zero-two" law for positive contractions of the
Banach-Kantorovich $L_p$-lattices.
 \vskip 0.3cm \noindent
{\it Mathematics Subject Classification}: 37A30, 47A35, 46B42, 46E30, 46G10.\\
{\it Key words and phrases}: Banach-Kantorovich $L_p$-lattice,
strong "zero-two" law, positive contraction.

\end{abstract}

\maketitle

\section{Introduction}

Starting from  von Neumman's \cite{vN} pioneering work, the
development of the theory of Banach bundles had been stimulated by
many works (see for example \cite{G1,G2}). There are many papers
were devoted to the applications of this theory to several
branches of analysis \cite{AAK2007,K1,K2,Wo}. Moreover, this
theory is well-connected with the theory of vector-valued Banach
spaces \cite{Gi,G1}, which has several applications (see for
example, \cite{LL}). In the present paper, we concentrate
ourselves to the theory of Banach bundles of $L_0$-valued Banach
spaces (see for more details \cite{Ga3,G1}). Note that such spaces
are called {\it Banach--Kantorovich spaces}. In \cite{G1,G2,K2})
the theory of Banach--Kantorovich spaces were developed. It is
known \cite{G1} that the theory of measurable bundles of Banach
lattices is sufficiently well explored. Therefore, it is natural
to employ methods of measurable bundles of such spaces to
investigate functional properties of Banach--Kantorovich spaces.
It is an effective tool which gives a good opportunity to obtain
various properties of these spaces \cite{Ga1,Ga2}. For example, in
\cite{GaC,Ga3} Banach-Kantorovich lattice $L_{p}({\nabla},{\mu})$
is represented as a measurable bundle of classical $L_{p}$
-lattices. Naturally, these functional Banach--Kantorovich spaces
have many similar properties like the classical ones, constructed
by the real valued measures. In \cite{CGa,GM2} this allowed to
establish several weighted ergodic theorems for positive
contractions of $L_{p}({\nabla},{\mu})$-spaces. In \cite{Ga2} the
convergence theorems of martingales on such lattices has been
proved. Some other applications of the measurable bundles of
Banach-Kantorovich spaces can be found in \cite{AAK2007,GM3}.

In \cite{OS} Ornstein and Sucheston  proved that, for any positive
contraction $T$ on an $L_1$-space, one has either
$\|T^n-T^{n+1}\|_1=2$ for all $n$ or
$\lim\limits_{n\to\infty}\|T^n-T^{n+1}\|_1=0$. An extension of
this result to positive operators on $L^{\infty}$-spaces was given
by Foguel \cite{F}. In \cite{Z1} Zahoropol generalized these
results, called {\it "zero-two" laws}, and  his result can be
formulated as follows:

\begin{thm}\label{A}  Let $T$ be a positive contraction of
$L_p$, $p>1,p\neq 2$. If  the following relation holds
$\big\||T^{m+1}-T^m|\big\|<2$ for some $m\in
{\mathbb{N}}\cup\{0\}$, then
$$
\lim_{n\to\infty}\|T^{n+1}-T^n\|=0.
$$
\end{thm}

In \cite{KT} this theorem was established for K\"{o}the spaces. In
particularly, from that result it follows the statement of the
theorem for a case $p=2$.

Furthermore,  the strong "zero-two" law for positive contractions
of $L_p$-spaces, $1 \leq p < +\infty$ was proved in \cite{W}. This
result is formulated as follows:

\begin{thm}\label{D}  Let $1\leq p < + \infty$ and $T$ be a
positive contraction of $L_p$. If $\big\||T^{m+1}-T^m|\big\|<2$
for some $m\in {\mathbb{N}}\cup\{0\}$, then
$$
\lim_{n\to\infty}\big\||T^{n+1}-T^n|\big\|=0. $$
\end{thm}

In \cite{GM} we have generalized Theorem \ref{A} for the positive
contractions of the Banach-Kantorovich $L_p$-lattices. Namely, the
following result was proved.

\begin{thm}\label{B} Let $T: L_p({\nabla},m)\to L_p({\nabla},m)$,
$p>1, p\neq 2$ be a positive linear contraction such that
$T\mathbf{1}\leq\mathbf{1}$. If one had
$\big\||T^{m+1}-T^m|\big\|<2\cdot\mathbf{1}$ for some
$m\in{\mathbb{N}}\cup\{0\}$. Then
$$
(o)-\lim_{n\to\infty}\|T^{n+1}-T^{n}\|=0.
$$
\end{thm}

The main aim of this paper is to prove the strong "zero-two" law
for the positive contractions of the Banach-Kantorovich lattices
$L_p(\nabla,m)$. To establish the main aim, we first study
majorizable operators acting on Banach-Kantorovich $L_p$-lattices
(see Section 3). Then using methods of measurable bundles of
Banach-Kantorovich lattices, in section 4 we prove the main result
of the present paper.

\section{Preliminaries}

Let $(\Omega,\Sigma,\mu)$ be a complete measure space with a
finite measure $\mu$. By $\mathcal{L}(\Omega)$ (resp.
$\mathcal{L}_\infty(\Omega)$ ) we denote the set of all (resp.
essentially bounded) measurable real functions defined on $\Omega$
a.e. By the standard way, we introduce an equivalence relation on
$\mathcal{L}(\Omega)$ by putting $f \sim g$ whenever $f = g$ a.e.
The set $L_0(\Omega)$ of all cosets $f^\sim = \{g \in
\mathcal{L}(\Omega): f \sim g\}$, endowed with the natural
algebraic operations, is an algebra with unit $\mathbf{1}(\omega)=
1$ over the field of reals $\br$. Moreover, with respect to the
partial order $f^\sim\leq g^\sim $ $\Leftrightarrow$ $f \leq g$
a.e., the algebra $L_0(\Omega)$ is a Dedekind complete Riesz space
with weak unit $\bf{1}$, and the set $B(\Omega) :=
B(\Omega,\Sigma,\mu)$ of all idempotents in $L_0(\Omega)$ is a
complete Boolean algebra. Furthermore, $L_\infty(\Omega) =
\{f^\sim: f\in \mathcal{L}_\infty(\Omega)\}$ is an order ideal in
$L_0(\Omega)$ generated by $\bf{1}$. In what follows, we will
write $f\in L_0(\Omega)$ instead of $f^\sim\in L_0(\Omega)$ by
assuming that the coset of $f$ is considered.

Let $E$ be a linear space over the real field $\mathbb{R}$. By
$\|\cdot\|$ we denote a $L_0(\Omega)$-valued norm on $E$. Then the
pair $(E,\|\cdot\|)$ is called a {\it lattice-normed space (LNS)
over $L_0(\Omega)$}. An LNS $E$ is said to be {\it
$d$-decomposable} if for every $x\in E$ and the decomposition
$\|x\|=f+g$ with $f$ and $g$ disjoint positive elements in
$L_0(\Omega)$ there exist $y,z\in E$ such that $x=y+z$ with
$\|y\|=f$, $\|z\|=g$.

Suppose that $(E,\|\cdot\|)$ is an LNS over $L_0(\Omega)$. A net
$\{x_\alpha\}$ of elements of $E$ is said to be {\it
$(bo)$-converging} to $x\in E$ (in this case we write
$x=(bo)$-$\lim x_\alpha$), if the net $\{\|x_\alpha - x\|\}$
$(o)$-converges to zero (here $(o)$-convergence means the order
convergence) in $L_0(\Omega)$ (written as $(o)$-$\lim \|x_\alpha
-x\|=0$). A net $\{x_\alpha\}_{\alpha\in A}$ is called {\it
$(bo)$-fundamental} if $(x_\alpha-x_\beta)_{(\alpha,\beta)\in
A\times A}$ $(bo)$-converges to zero.

An LNS in which every $(bo)$-fundamental net $(bo)$-converges is
called {\it $(bo)$-complete}. A {\it Banach-Kantorovich space
(BKS) over $L_0(\Omega)$} is a $(bo)$-complete $d$-decomposable
LNS over $L_0(\Omega)$. It is well known \cite{K1},\cite{K2} that
every BKS $E$ over $L_0(\Omega)$ admits an $L_0(\Omega)$-module
structure such that $\|fx\|=|f|\cdot\|x\|$ for every $x\in E,\
f\in L_0(\Omega)$, where $|f|$ is the modulus of a function $f\in
L_0(\Omega)$.
 A BKS $({\mathcal{U}},\|\cdot\|)$ is called a {\it
Banach-Kantorovich lattice} if  ${\mathcal{U}}$ is a vector
lattice and the norm $\|\cdot\|$ is monotone, i.e.
$|u_1|\leq|u_2|$ implies $\|u_1\|\leq\|u_2\|$. It is known
\cite{K1} that the cone ${\mathcal{U}}_+$ of  positive elements is
$(bo)$-closed.

 Let $\nabla$  be an arbitrary complete Boolean algebra and let
$X(\nabla)$ be the Stone space of $\nabla$. Assume that
$L_0(\nabla) := C_\infty(X(\nabla))$ be the algebra of all
continuous functions $x : X(\nabla)\rightarrow [-\infty,+\infty]$
that take the values $\pm\infty$ only on nowhere dense subsets of
$X(\nabla)$. Finally, by $C(X(\nabla))$ we denote the subalgebra
of all continuous real functions on $X(\nabla)$.

Given a complete Boolean algebra $\nabla$, let us consider a
mapping $m : {\nabla}\to L_0(\Omega)$. Such a mapping is called a
\textit{$L_0(\Omega)$-valued measure} if one has
\begin{enumerate}
\item[(i)]  $m(e)\geq 0$ for all $e\in{\nabla}$ and
$m(e)=0\Leftrightarrow e=0$; \item[(ii)]
 $m(e\vee g)=m(e)+m(g)$ if $\ e\wedge g=0, e,g\in{\nabla}$;

\item[(iii)] $m(e_\alpha)\downarrow 0$ for any net
$e_\alpha\downarrow 0$.
\end{enumerate}

Following  the well-known scheme of the construction of
$L_p$-spaces,  a space $L_p(\nabla,m)$ can be defined by
$$
L_p(\nabla,{m})=\left\{f\in L_0(\nabla): |f|_p:=\int| f|^pd{m} -
\textrm{exist} \ \right\}, \ \ \ p\geq 1
$$
where $m$ is a  $L_0(\Omega)$-valued measure on $\nabla$.

 A $L_0(\Omega)$-valued measure $m$ is said  to be  \textit{disjunctive
 decomposable ($d$-decomposable)}, if
for every $e\in \nabla$ and the decomposition $m(e)=a_1+a_2, \quad
a_1\wedge a_2=0, \quad a_i\in L_0(B)$ there exit $e_1,e_2\in
\nabla$ such that $e=e_1\vee e_2$ and $m(e_i)=a_i, i=1,2.$

\begin{thm}\label{22}\cite{Ga3} The following statements hold:
\begin{enumerate}
\item[(i)] The pair $(L_p(\nabla,m), |\cdot|_p)$ is
$(bo)$--complete lattice. Moreover, it is an ideal linear subspace
of $L_0(\nabla)$, i.e. from $|x|\leq |y|$, $y\in L_p(\nabla,m)$,
$x\in L_0(\nabla)$ it follows that $x\in L_p(\nabla,m)$ and
$|x|_p\leq|y|_p$;\\
\item[(ii)] If $0\leq x_\alpha\in L_p(\nabla,m)$ and
$x_\alpha\downarrow 0$, then $| x_\alpha|_p\downarrow 0$;\\
\item[(iii)] If the measure $m$ is $d$-decomposable, then $|
\alpha)x|_p=|\alpha|| x|_p$ for
 all $\alpha\in L_0(\Omega), x\in L_p(\nabla,m)$;\\

\item[(iv)] If the measure $m$ is $d$-decomposable, then
$(L_p(\nabla,m), |\cdot|_p)$
 is a Banach~---Kantorovich space;\\
\item[(v)] One has $L_\infty(\nabla,m):=C(X(\nabla))\subset
L_p(\nabla,m)\subset
 L_q(\nabla,m)$, $1\leq q\leq p.$ Moreover, $L_\infty(\nabla,m)$
 is  $(bo)$--dense in $(L_1(\nabla,m), \|\cdot\|_1).$
\end{enumerate}
\end{thm}

Now we mention necessary facts  from the theory of measurable
bundles of Boolean algebras and Banach spaces (see \cite{G1} for
more details).

Let $(\Omega,\Sigma,\mu)$ be the same as above and  $X$ be a
mapping assigning an $L_p$-space constructed by a real-valued
measure $m_\omega$, i.e. $L_p(\nabla_\omega,m_\omega)$ to each
point $\omega\in\Omega$ and let
$$
L=\bigg\{\sum\limits_{i=1}^n \alpha_i e_i : \alpha_i\in
{\mathbb{R}}, \ \ e_i(\omega)\in\nabla_{\omega},\
i=\overline{1,n},\ n\in\mathbb{N}\bigg\}$$ be a set of sections.
In \cite{Ga3} it has been established that the pair $(X,L)$ is a
measurable bundle of Banach lattices and $L_0(\Omega,X)$ is modulo
ordered isomorphic to $L_p(\nabla,\mu)$.

Let $\rho$ be a lifting in $L_\infty(\Omega)$ (see \cite{G1}). As
before, let $\nabla$ be an arbitrary complete Boolean subalgebra
of $\nabla(\Omega)$ and $m$ be an $L_0(\Omega)$-valued measure on
$\nabla$. By $L_\infty(\nabla,m)$ we denote the set of all
essentially bounded functions w.r.t. $m$ taken from $L_0(\nabla)$.

A mapping $\ell : L_\infty(\nabla,m)(\subset L_{\infty}(\Omega,X))
\to {\mathcal{L}}_\infty (\Omega,X)$  is called a  {\it
vector-valued lifting} \cite{G1} associated with the lifting
$\rho$ if it satisfies the following conditions:
\begin{enumerate}
   \item[(1)] $\ell(\hat{u})\in\hat{u}$ for all $\hat{u}$ such that
   $dom(\hat{u})=\Omega$;
   \item[(2)]
   $\|\ell(\hat{u})\|_{L_p(\nabla_{\omega},m_\omega)}=\rho(|\hat{u}|_p)(\omega)$;
   \item[(3)] $\ell(\hat{u}+\hat{v})=\ell(\hat{u})+\ell(\hat{v})$
   for every $\hat{u},\hat{v}\in L_\infty(\nabla,m)$;
   \item[(4)] $\ell(h\cdot\hat{u})=\rho(h)\ell(\hat{u})$ for every
   $\hat{u}\in L_\infty(\nabla,m),\ h\in L_\infty(\Omega)$;
   \item[(5)] $\ell(\hat{u})\geq 0$ whenever $\hat{u}\geq 0$;
   \item[(6)] the set $\{\ell(\hat{u})(\omega) : \hat{u}\in
    L_\infty(\nabla,m)\}$ is dense in $X(\omega)$ for all $\omega\in\Omega$;
   \item[(7)]
   $\ell(\hat{u}\vee\hat{v})=\ell(\hat{u})\vee\ell(\hat{v})$
    for every $\hat{u},\hat{v}\in L_\infty(\nabla,m)$.
\end{enumerate}

In \cite{Ga3} the existence of the vector-valued lifting was
proved.

Let $L_p({\nabla},{m})$ ($p\geq 1$) be a Banach-Kantorovich
lattice. A linear mapping $T:L_p({\nabla},{m})\to
L_p({\nabla},{m})$ is called {\it positive} if $T\hat{f}\geq 0$
whenever $\hat{f}\geq 0$. We say that $T$ is a {\it
$L_0(\Omega)$-bounded mapping} if there exists a function $k\in
L_0(\Omega)$ such that $|T\hat{f}|_p\leq k |\hat{f}|_p$ for all
$\hat{f}\in L_p(\nabla,\mu)$. For such a mapping we can define an
element of $L_0(\Omega)$ as follows
$$
\|T\| =\sup\limits_{|\hat{f}|_p\leq\mathbf{1}} |T\hat{f}|_p,
$$
which is called an {\it $L_0(\Omega)$-valued norm} of $T$. A
mapping $T$ is said to be a {\it contraction} if one has
$\|T\|\leq\mathbf{1}$. Some examples of contractions can be found
in \cite{GM2}.

In the sequel we will need the following bundle representation of
$L_0(\Omega)$-linear $L_0(\Omega)$-bounded operators acting in
Banach-Kantorovich lattices.

\begin{thm}\label{21}\cite{GM} Let $L_p({\nabla},{m})$ ($p\geq 1$) be a Banach-Kantorovich
lattice, and $L_p(\nabla_\omega,m_\omega)$ be the corresponding
$L_p$-spaces constructed by real valued measures.  Let $T :
L_p({\nabla},{m})\to L_p({\nabla},{m})$ be a positive linear
contraction such that $T\mathbf{1}\leq\mathbf{1}$. Then for every
$\omega\in\Omega$ there exists a positive contraction $T_\omega :
L_p(\nabla_\omega,\mu_\omega)\to L_p(\nabla_\omega,m_\omega)$ such
that $T_\omega f(\omega) = (T\hat{f})(\omega)$ a.e. for every
$\hat{f}\in L_p({\nabla},{m})$.
\end{thm}

\section{Majorizable operators in Banach-Kantorovich $L_p$-lattices}

In this section, we are going to study majorizable operators in
Banach-Kantorovich $L_p$-lattices.

\begin{thm}\label{31} Let $T: L_1(\nabla,m)\rightarrow
 L_1(\nabla,m)$ be an $L_0(\Omega)$- bounded linear operator in
 Banach-Kantorovich lattice $L_1(\nabla,m)$. Then there exists a unique $|T|$- $L_0(\Omega)$- bounded linear operator in
 $L_1(\nabla,m)$ such that
\begin{enumerate}
\item[(a)] $\|T\|=\| |T| \|;$

\item[(b)] one has $ |T\hat{f}|\leq |T||\hat{f}|$, for all
$\hat{f}\in L_1(\nabla,m)$;

\item[(c)] for each $\hat{f}\in L_1(\nabla,m)$ with $\hat{f}\geq0$
one has $|T|\hat{f}=\sup\{|T\hat{g}|: \hat{g}\in L_1(\nabla,m),
|\hat{g}|\leq \hat{f}\ \};$

\item[(d)] $\|T\|_\infty=\||T|\|_\infty.$
\end{enumerate}
\end{thm}

\begin{proof} Let $\mathcal{P}$ denote the family of all finite
measurable partitions $\pi = \{B_1, B_2,\dots, B_m\}$ of $\Omega$.
We partially order $\mathcal{P}$ in the usual way, i.e. for  $\pi
= \{B_1, B_2,\dots, B_m\}$  and $\pi' = \{B'_1, B'_2,\dots,
B'_k\}$  we write $\pi\leq \pi^{'}$ if $\pi^{'}$ is a refinement
of $\pi$, i.e. each set $B_i$ is a union of sets $\{B'_i\}$.

Given $\pi\in \mathcal{P}$, and for every $\hat{f}\in
L_1({\nabla},m), \hat{f}\geq0$ we define
$$
T_{\pi}\hat{f}\:=\sum\limits_{i=1}^m  |T(\chi_{B_i}\hat{f})|.
$$
Clearly $\pi\leq \pi'$ implies $T_{\pi}\hat{f}\leq
T_{\pi'}\hat{f}$. From $ |\hat{f}|_1=\sum\limits_{i=1}^m
|\chi_{B_i}\hat{f}|_1$ we obtain $|T_{\pi}\hat{f} |_1\leq
\|T\||\hat{f}|_1$.  Since $\{T_{\pi}\hat{f}: \pi\in \mathcal{P}\}$
is increasing on $\mathcal{P}$ and is norm bounded, therefore one
can define

$$ |T|\hat{f}:= \lim\limits_{\pi\in \mathcal{P}}T_{\pi}\hat{f},
\quad \hat{f}\geq0.$$

 We clearly have
\begin{equation}\label{11}
| |T|\hat{f}|_1\leq \|T\||\hat{f}|_1, \hat{f}\geq0
\end{equation}
and $|T|$ is linear on positive functions. Therefore  $|T|$ can be
extended by the linearity to whole $L_1(\nabla,m)$. This extension
is again denoted by $|T|$.

For $\hat{f}\geq0$ and $|\hat{g}|\leq\hat{f}$ we obtain $
|T|\hat{f}\geq|T\hat{g}|$ by means of the approximation argument
with simple functions. This yields (b).

(c).  From (b) we have  $|T|\hat{g}\geq|T\hat{g}|$, i.e. $T$ has a
positive majorant. Then by \cite[Theorem VIII 1.1]{V} $T$ is
regular. Hence, using \cite[formula (10),p.231]{V} one finds
$|T|\hat{f}=\sup\{|T\hat{g}|: \hat{g}\in L_1(\nabla,m),
|\hat{g}|\leq \hat{f}\ \}.$

(a). Again from (b) we get $\|T\|\leq \| |T| \|$ and by \eqref{11}
one finds $\| |T| \|\leq \|T\|$. Hence, $\|T\|= \| |T| \|$.

(d). Let $\hat{f}\in L^\infty(\hat{\nabla},\hat{\mu}) $. It is
then clear that from $|T\hat{f}|\leq |T||\hat{f}|$ one gets
$\|T\|_\infty\|\hat{f}\|_\infty\leq\||T|\|_\infty\|\hat{f}\|_\infty$
which means $\|T\|_\infty\leq\||T|\|_\infty.$

Using (c) we obtain
$$
|T||\hat{f}|=\sup\limits_{|\hat{g}|\leq|\hat{f}|}|T\hat{g}|\leq\sup\limits_{|\hat{g}|\leq|\hat{f}|}\|T\|_\infty\|\hat{g}\|_\infty{\mathbf{1}}\leq
\|T\|_\infty\|\hat{f}\|_\infty{\mathbf{1}}.$$ Hence,
$\||T|\|_\infty\leq\|T\|_\infty$ and
$\||T|\|_\infty=\|T\|_\infty.$
\end{proof}

\begin{defin} A linear operator $A: L_p(\nabla,m)\rightarrow
 L_p(\nabla,m)$ is called \textit{majorizable} if there exists an $L_0(\Omega)$- bounded positive linear operator $S: L_p(\nabla,m)\rightarrow
 L_p(\nabla,m)$ such that
 $$|A\hat{f}|\leq S(|\hat{f}|)$$
 for all $\hat{f}\in L_p(\nabla,m).$ The operator $S$ is called
 \textit{majorant}.
\end{defin}

\begin{thm}\label{32} Let $T: L_p(\nabla,m)\rightarrow
 L_p(\nabla,m)$ be a majorizable operator with a majorant $S$ on Banach-Kantorovich lattice
 $L_p(\nabla,m)$. Then there exists a unique   $|T|$- $L_0(\Omega)$- bounded linear operator on $L_p(\nabla,m)$ such that
\begin{enumerate}
\item[(a)] $\||T|\|\leq\| S \|;$

\item[(b)] one has $|T\hat{f}|\leq |T||\hat{f}|$,  for all
$\hat{f}\in L_p(\nabla,m)$;

\item[(c)] for each $\hat{f}\in L_p(\nabla,m), \hat{f}\geq0$ one
has $|T|\hat{f}=\sup\{|T\hat{g}|: \hat{g}\in L_p(\nabla,m),
|\hat{g}|\leq \hat{f} \};$
\end{enumerate}
\end{thm}

\begin{proof} The proof of the existence of $|T|$ and (b), (c) are similar to
the proof of Theorem \ref{31}.  Now we prove (a). From

$$|T|\hat{f}=\sup\{|T\hat{g}|: \hat{g}\in L_p(\nabla,m),
|\hat{g}|\leq \hat{f}\}\leq \sup\{S|\hat{g}|: \hat{g}\in
L_p(\nabla,m), |\hat{g}|\leq \hat{f}\}=S\hat{f}$$ we get
$$||T|\hat{f}|_p\leq |S\hat{f}|_p\leq\|S\|\|\hat{f}|_p$$
hence
$$\||T|\|\leq\|S\|.$$ This completes the proof.
\end{proof}

\begin{thm}\label{33} If $A: L_p(\nabla,m)\rightarrow
 L_p(\nabla,m)$ is a majorizable operator, and its majorant $S$ is a contraction with $S\mathbf{1}\leq\mathbf{1}$, then for every $\omega\in\Omega$
 there exists a majorizable operator $A_\omega: L_p(\nabla_\omega,m_\omega)\rightarrow
 L_p(\nabla_\omega,m_\omega)$ such that
 $$A_\omega f(\omega)=(A\hat{f})(\omega) \ \ \textrm{a.e.}$$
 for all $\hat{f}\in L_p(\nabla,m).$
\end{thm}

\begin{proof} Since  $S$ is
a contraction and $S\mathbf{1}\leq\mathbf{1}$, we obtain that
$A(L_\infty(\nabla,m))\subset L_\infty(\nabla,m)$.

 Now we define a linear operator
 $\varphi_\omega$ from  $\{\ell(\hat{f})(\omega) : \hat{f}\in
 L_\infty(\nabla,m)\}$ into
 $L_p(\nabla_\omega,m_\omega)$ by
 $$\varphi_\omega
 (\ell(\hat{f})(\omega))=\ell(A\hat{f})(\omega)$$ where $\ell$ is
 the vector lifting of $L_\infty(\nabla,m)$
 associated with the lifting  $\rho$.

 From the majorizability of $A$ one gets
 $$|\varphi(\omega)
 (\ell(\hat{f})(\omega))|=|\ell(A\hat{f})(\omega)|=\ell(|A\hat{f}|)(\omega)\leq\ell(S|\hat{f}|)(\omega)
 =S'_\omega(\ell(|\hat{f}|)(\omega))=S'_\omega(|\ell(|\hat{f}|)(\omega)|)$$
 for any positive $\hat{f}\in L_\infty(\nabla,m)$, where
 $S'_\omega$ is a positive contraction on $\{\ell(\hat{f})(\omega) : \hat{f}\in
 L_\infty(\nabla,m)\}$. This means that $\varphi(\omega)$ is
 a majorizable operator on $\{\ell(\hat{f})(\omega) : \hat{f}\in
 L_\infty(\nabla,m)\}$.

 From $|S\hat{f}|_p\leq |\hat{f}|_p$ we obtain
 $$\|\ell(A\hat{f})(\omega)\|_{{L_p(\nabla_\omega,m_\omega)}}=
\rho(|A\hat{f}|_p)(\omega)\leq \rho(|S\hat{f}|_p)(\omega)\leq
\rho(|\hat{f}|_p)(\omega)=
 \|\ell(\hat{f})(\omega)\|_{{L_p(\nabla_\omega,m_\omega)}}$$ which
 implies  that $\varphi_\omega$ and $S'_\omega$ are well defined and
 bounded. Moreover, $S'_\omega$ is positive (see Theorem \ref{21}).
Due to the density of $\{\ell(\hat{f})(\omega) : \hat{f}\in
 L_\infty(\nabla,m)\}$ in
 $L_p(\nabla_\omega,m_\omega)$, we can extend $\varphi_\omega$ and
 $S'_\omega$, respectively,  to $L_p(\nabla_\omega,m_\omega)$. We respectively denote the extensions
by $A_\omega$ and $S_\omega$. One can see that $A_\omega$ is
bounded, and $S_\omega$ is positive bounded.

From $$|\varphi(\omega)
 (\ell(\hat{f})(\omega))|\leq S'_\omega(|\ell(\hat{f})(\omega)|)$$
 for any $\hat{f}\in
 L_\infty(\nabla,m)$ one finds
$$|A_\omega
 (f(\omega))|\leq S_\omega(|f(\omega)|)$$ i.e. $A_\omega$ is
 majorizable.

Repeating the argument of the proof of \cite[Theorem 2.1]{GM}, we
can prove that
$$A_\omega f(\omega)=(A\hat{f})(\omega)$$
 for almost all $\omega\in\Omega$ and for all $\hat{f}\in
 L_p(\nabla,m)$. This completes the proof.
\end{proof}

\begin{thm}\label{34} If $A: L_p(\nabla,m)\rightarrow
 L_p(\nabla,m)$ is a majorizable operator, and its majorant $S$ is a contraction with $S\mathbf{1}\leq\mathbf{1}$, then
 $$\||A|_\omega \|_{p,\omega}=\||A_\omega|\|_{p,\omega}$$
 for almost all $\omega\in\Omega$, where $\|\cdot\|_{p,\omega}$ is the norm of an operator from
$L_p(\nabla_\omega,m_\omega)$ to $L_p(\nabla_\omega,m_\omega)$.
\end{thm}

\begin{proof} Due to  $-|A|\leq
A\leq|A|$ we have $-|A|_\omega\leq A_\omega\leq|A|_\omega$ which
yields $|A_\omega|\leq |A|_\omega$ for almost all
$\omega\in\Omega$. Hence, $\||A|_\omega
\|_{p,\omega}\geq\||A_\omega|\|_{p,\omega}$
 for almost
all $\omega\in\Omega$.

 Let $\{\pi_n\}$ be an increasing sequence  in $\mathcal{P}$ such that
$|A|\hat{f}=(bo)-\lim\limits_{n\rightarrow\infty}A_{\pi_n}\hat{f}$,
for $ 0\leq\hat{f}\in L_p(\nabla,m)$.

One can see that
\begin{equation}\label{22}
(A_{\pi_n}\hat{f})(\omega)=\sum\limits_{i=1}^{m}|A(\chi_{B_i}\hat{f})|(\omega)=
\sum\limits_{i=1}^{m}|A_\omega(\chi_{B_i}(\omega){f})(\omega)|=A_{\omega,\pi_n}f(\omega)
\end{equation} for almost all $\omega\in\Omega$.

Now using
$$|A|\hat{f}=(bo)-\lim\limits_{n\rightarrow\infty}A_{\pi_n}\hat{f} \ \ \textrm{in} \ \ L_p({\nabla},m),$$
with \eqref{22} we obtain
$|A_{\pi_n}\hat{f}|_p\stackrel{(o)}\rightarrow
\big||A|\hat{f}\big|_p$ or
$|A_{\pi_n}\hat{f}|_p(\omega)\rightarrow
\big||A|\hat{f}\big|_p(\omega)$ for almost all $\omega\in\Omega$.
Hence,
$$
\|A_{\pi_n,\omega}{f}(\omega)\|_{L_p(\nabla_\omega,m_\omega)}\rightarrow
\big\||A|_\omega{f}(\omega)\big\|_{L_p(\nabla_\omega,m_\omega)}$$
for almost all $\omega\in\Omega$.

On the other hand, one has
$$\lim\limits_{n\rightarrow\infty}\|A_{\pi_n,\omega}{f}(\omega)\|_{L_p(\nabla_\omega,m_\omega)}\leq
\big\||A_\omega|{f}(\omega)\big\|_{L_p(\nabla_\omega,m_\omega)}$$
for almost all $\omega\in\Omega$. This means that
$$\big\||A|_\omega{f}(\omega)\big\|_{L_p(\nabla_\omega,m_\omega)}\leq\big\||A_\omega|{f}(\omega)\big\|_{L_p(\nabla_\omega,m_\omega)}$$
or
$$\big\||A|_\omega\big\|_{p,\omega}\leq\big\||A_\omega|\big\|_{p,\omega}$$
for almost all $\omega\in\Omega$. Hence
$$\big\||A|_\omega\big\|_{p,\omega}=\big\||A_\omega|\big\|_{p,\omega}$$
for almost all $\omega\in\Omega$. This completes the proof.
\end{proof}

\section{The strong "zero-two" law}

In this section we are going to prove an analog of the strong
"zero-two" law for positive contractions in the Banach-Kantorovich
$L_p$-lattices. Before the formulation of the main result, we need
some auxiliary results.

\begin{prop}\label{41} Let $T, S : L_p({\nabla},m)\to
L_p({\nabla},m)$  be two positive linear contractions such that
$T\mathbf{1}\leq\mathbf{1}$, $S\mathbf{1}\leq\mathbf{1}$. Then
$$
\big\||T_\omega-S_\omega|\big\|_{p,\omega}\geq
\big\||T-S|\big\|(\omega), \ \ \textrm{a.e.}
$$
here $|\cdot|$ means the modulus of an operator.
\end{prop}

\begin{proof} Due to $(T-S)(\hat{f})\leq
T(\hat{f})$ for any positive $\hat{f}\in L_p(\nabla,m)$ one gets
$$|(T-S)(\hat{f})|\leq T(|\hat{f}|)$$
for any  $\hat{f}\in L_p(\nabla,m)$. Hence $T-S$  is majorizable.
Since $T$ is a contraction and $T\mathbf{1}\leq\mathbf{1}$ by
Theorem \ref{34}, we obtain
$\big\||T-S|_\omega\big\|_{p,\omega}=\big\||T_{\omega}-S_{\omega}|\big\|_{p,\omega}$
for almost all $\omega\in\Omega$. By \cite[Proposition 2]{GK} for
any $\varepsilon>0$ there exists $\hat{f}\in L_p(\nabla,m)$ with
$|\hat{f}|_p=\mathbf{1}$ such that

$$\big\||T-S|\big\|-\varepsilon{\mathbf{1}}\leq\big ||T-S|\hat{f}\big|_p.$$
Then
\begin{eqnarray*}
\big\||T-S|\big\|(\omega)-\varepsilon{\mathbf{1}}&\leq&
\big||T-S|\hat{f}\big|_p(\omega)=\|(|T-S|\hat{f})(\omega)\|_{L_p(\nabla_\omega,m_\omega)}\\[2mm]
&=& \||T-S|_\omega f(\omega)\|_{L_p(\nabla_\omega,m_\omega)}\leq
\big\||T-S|_\omega\big\|_{p,\omega}\\[2mm]
&=&\big\||T_{\omega}-S_{\omega}|\big\|_{p,\omega}
\end{eqnarray*}
for almost all $\omega\in\Omega$. The arbitrariness of
$\varepsilon>0$  implies the statement. \end{proof}

\begin{cor}\label{42} Let $T,S : L_p({\nabla},m)\to L_p({\nabla},m)$
be two positive linear contractions such that
$T\mathbf{1}\leq\mathbf{1}$, $S\mathbf{1}\leq\mathbf{1}$. Then
$$
\big\||T_\omega-S_\omega|\big\|_{p,\omega}=
\big\||T-S|\big\|(\omega), \ \ \textrm{a.e.}
$$
\end{cor}

The proof follows from \cite[Proposition 3.2]{GM} and Proposition
\ref{41}.

The next theorem is our main result of the present paper.

\begin{thm} Let $T : L_p({\nabla},m)\to L_p({\nabla},m)$
be a positive linear contraction such that
$T\mathbf{1}\leq\mathbf{1}$. If one has
$\big\||T^{m+1}-T^m|\big\|<2\cdot\mathbf{1}$ for some
$m\in{\mathbb{N}}\cup\{0\}$. Then
$$
(o)-\lim_{n\to\infty}\big\||T^{n+1}-T^{n}|\big\|=0.
$$
\end{thm}

\begin{proof} From Corollary \ref{42} it follows that
$$
\big\||T^{m+1}_\omega-T^{m}_\omega|\big\|_{p,\omega}=
\big\||T^{m+1}-T^{m}|\big\|(\omega), \ \ \textrm{a.e.}
$$ on $\Omega.$ Therefore, due to $\big\||T^{m+1}-T^m|\big\|<2\cdot\mathbf{1}$ for
some $m\in{\mathbb{N}}\cup\{0\}$ we find
$\big\||T^{m+1}_\omega-T^{m}_\omega|\big\|_{p,\omega}<2$ for
almost all $\omega\in\Omega$. According to Theorem \ref{21} we
conclude that $T_\omega$ is a  positive contraction on
$L_p(\nabla_\omega,m_\omega).$ Hence, the contraction $T_\omega$
satisfies the conditions of Theorem \ref{D} for almost all
$\omega\in\Omega$, which yields that
$$\lim\limits_{n\rightarrow\infty}\big\||T_{\omega}^{n+1}-T_{\omega}^{n}|\big\|=0$$
for almost all $\omega\in\Omega$. Then again using Corollary
\ref{42} we obtain
$$\lim\limits_{n\rightarrow\infty}\big\||T^{n+1}-T^{n}|\big\|(\omega)=0$$
for almost all $\omega\in\Omega$. Therefore,
$$
(o)-\lim_{n\to\infty}\big\||T^{n+1}-T^{n}|\big\|=0.
$$
This completes the proof.
\end{proof}

\section*{Acknowledgement} The first  author  acknowledges  the MOE Grant FRGS13-071-0312. The second named author thanks the MOE grant
FRGS14-135-0376, and the Junior Associate scheme of the Abdus
Salam International Centre for Theoretical Physics, Trieste,
Italy.

\end{document}